\documentclass[11pt, a4paper, twoside,leqno]{amsart}
\usepackage[centering, totalwidth = 380pt, totalheight = 580pt]{geometry}
\usepackage{amssymb, amsmath, amsthm}
\usepackage{enumerate, microtype, stmaryrd, url} 
\usepackage[latin1]{inputenc}
\usepackage[arrow, matrix, tips, curve]{xy}
 \usepackage{color}
\definecolor{darkgreen}{rgb}{0,0.45,0}
\usepackage[pagebackref,colorlinks,citecolor=darkgreen,linkcolor=darkgreen]{hyperref}
\SelectTips{cm}{10}

 \DeclareMathOperator{\ob}{ob}

\newcommand{\cat}[1]{\mathbf{#1}}
\newcommand{\op}{\mathrm{op}}

\newcommand{\thg}{{\mathord{\text{--}}}}

\newcommand{\spn}[1]{{\langle{#1}\rangle}}

\newcommand{\cd}[2][]{\vcenter{\hbox{\xymatrix#1{#2}}}}

\renewcommand{\phi}{\varphi}
\newcommand{\A}{{\mathcal A}}
\newcommand{\B}{{\mathcal B}}
\newcommand{\C}{{\mathcal C}}
\newcommand{\D}{{\mathcal D}}

\renewcommand{\L}{{\mathcal L}}
\newcommand{\M}{{\mathcal M}}
\newcommand{\N}{{\mathcal N}}

\renewcommand{\P}{{\mathcal P}}

\newcommand{\Ss}{{\mathcal S}}
\newcommand{\T}{{\mathcal T}}

\newcommand{\V}{{\mathcal V}}
\newcommand{\W}{{\mathcal W}}

\newcommand{\xtor}[1]{\cdl[@1]{{} \ar[r]|-{\object@{|}}^{#1} & {}}}

\makeatletter

\def\hookleftarrowfill@{\arrowfill@\leftarrow\relbar{\relbar\joinrel\rhook}}
\def\twoheadleftarrowfill@{\arrowfill@\twoheadleftarrow\relbar\relbar}
\def\leftbararrowfill@{\arrowdoublefill@{\leftarrow\mkern-5mu}\relbar\mapstochar\relbar\relbar}
\def\Leftbararrowfill@{\arrowdoublefill@{\Leftarrow\mkern-2mu}\Relbar\Mapstochar\Relbar\Relbar}
\def\leftringarrowfill@{\arrowdoublefill@{\leftarrow\mkern-3mu}\relbar{\mkern-3mu\circ\mkern-2mu}\relbar\relbar}
\def\lefttriarrowfill@{\arrowfill@{\mathrel\triangleleft\mkern0.5mu\joinrel\relbar}\relbar\relbar}
\def\Lefttriarrowfill@{\arrowfill@{\mathrel\triangleleft\mkern1mu\joinrel\Relbar}\Relbar\Relbar}

\def\hookrightarrowfill@{\arrowfill@{\lhook\joinrel\relbar}\relbar\rightarrow}
\def\twoheadrightarrowfill@{\arrowfill@\relbar\relbar\twoheadrightarrow}
\def\rightbararrowfill@{\arrowdoublefill@{\relbar\mkern-0.5mu}\relbar\mapstochar\relbar\rightarrow}
\def\Rightbararrowfill@{\arrowdoublefill@{\Relbar\mkern-2mu}\Relbar\Mapstochar\Relbar\Rightarrow}
\def\rightringarrowfill@{\arrowdoublefill@\relbar\relbar{\mkern-2mu\circ\mkern-3mu}\relbar{\mkern-3mu\rightarrow}}
\def\righttriarrowfill@{\arrowfill@\relbar\relbar{\relbar\joinrel\mkern0.5mu\mathrel\triangleright}}
\def\Righttriarrowfill@{\arrowfill@\Relbar\Relbar{\Relbar\joinrel\mkern1mu\mathrel\triangleright}}

\def\leftrightarrowfill@{\arrowfill@\leftarrow\relbar\rightarrow}
\def\mapstofill@{\arrowfill@{\mapstochar\relbar}\relbar\rightarrow}

\renewcommand*\xleftarrow[2][]{\ext@arrow 20{20}0\leftarrowfill@{#1}{#2}}
\providecommand*\xLeftarrow[2][]{\ext@arrow 60{22}0{\Leftarrowfill@}{#1}{#2}}
\providecommand*\xhookleftarrow[2][]{\ext@arrow 10{20}0\hookleftarrowfill@{#1}{#2}}
\providecommand*\xtwoheadleftarrow[2][]{\ext@arrow 60{20}0\twoheadleftarrowfill@{#1}{#2}}
\providecommand*\xleftbararrow[2][]{\ext@arrow 10{22}0\leftbararrowfill@{#1}{#2}}
\providecommand*\xLeftbararrow[2][]{\ext@arrow 50{24}0\Leftbararrowfill@{#1}{#2}}
\providecommand*\xleftringarrow[2][]{\ext@arrow 10{26}0\leftringarrowfill@{#1}{#2}}
\providecommand*\xlefttriarrow[2][]{\ext@arrow 80{24}0\lefttriarrowfill@{#1}{#2}}
\providecommand*\xLefttriarrow[2][]{\ext@arrow 80{24}0\Lefttriarrowfill@{#1}{#2}}

\renewcommand*\xrightarrow[2][]{\ext@arrow 01{20}0\rightarrowfill@{#1}{#2}}
\providecommand*\xRightarrow[2][]{\ext@arrow 04{22}0{\Rightarrowfill@}{#1}{#2}}
\providecommand*\xhookrightarrow[2][]{\ext@arrow 00{20}0\hookrightarrowfill@{#1}{#2}}
\providecommand*\xtwoheadrightarrow[2][]{\ext@arrow 03{20}0\twoheadrightarrowfill@{#1}{#2}}
\providecommand*\xrightbararrow[2][]{\ext@arrow 01{22}0\rightbararrowfill@{#1}{#2}}
\providecommand*\xRightbararrow[2][]{\ext@arrow 04{24}0\Rightbararrowfill@{#1}{#2}}
\providecommand*\xrightringarrow[2][]{\ext@arrow 01{26}0\rightringarrowfill@{#1}{#2}}
\providecommand*\xrighttriarrow[2][]{\ext@arrow 07{24}0\righttriarrowfill@{#1}{#2}}
\providecommand*\xRighttriarrow[2][]{\ext@arrow 07{24}0\Righttriarrowfill@{#1}{#2}}

\providecommand*\xmapsto[2][]{\ext@arrow 01{20}0\mapstofill@{#1}{#2}}
\providecommand*\xleftrightarrow[2][]{\ext@arrow 10{22}0\leftrightarrowfill@{#1}{#2}}
\providecommand*\xLeftrightarrow[2][]{\ext@arrow 10{27}0{\Leftrightarrowfill@}{#1}{#2}}

\makeatother

\newcommand{\twocong}[2][0.5]{\ar@{}[#2] \save ?(#1)*{\cong}\restore}
\newcommand{\twoeq}[2][0.5]{\ar@{}[#2] \save ?(#1)*{=}\restore}
\newcommand{\rtwocell}[3][0.5]{\ar@{}[#2] \ar@{=>}?(#1)+/l 0.2cm/;?(#1)+/r 0.2cm/^{#3}}
\newcommand{\ltwocell}[3][0.5]{\ar@{}[#2] \ar@{=>}?(#1)+/r 0.2cm/;?(#1)+/l 0.2cm/^{#3}}
\newcommand{\ltwocello}[3][0.5]{\ar@{}[#2] \ar@{=>}?(#1)+/r 0.2cm/;?(#1)+/l 0.2cm/_{#3}}
\newcommand{\dtwocell}[3][0.5]{\ar@{}[#2] \ar@{=>}?(#1)+/u  0.2cm/;?(#1)+/d 0.2cm/^{#3}}
\newcommand{\dltwocell}[3][0.5]{\ar@{}[#2] \ar@{=>}?(#1)+/ur  0.2cm/;?(#1)+/dl 0.2cm/^{#3}}
\newcommand{\drtwocell}[3][0.5]{\ar@{}[#2] \ar@{=>}?(#1)+/ul  0.2cm/;?(#1)+/dr 0.2cm/^{#3}}
\newcommand{\dthreecell}[3][0.5]{\ar@{}[#2] \ar@3{->}?(#1)+/u  0.2cm/;?(#1)+/d 0.2cm/^{#3}}
\newcommand{\utwocell}[3][0.5]{\ar@{}[#2] \ar@{=>}?(#1)+/d 0.2cm/;?(#1)+/u 0.2cm/_{#3}}
\newcommand{\dtwocelltarg}[3][0.5]{\ar@{}#2 \ar@{=>}?(#1)+/u  0.2cm/;?(#1)+/d 0.2cm/^{#3}}
\newcommand{\utwocelltarg}[3][0.5]{\ar@{}#2 \ar@{=>}?(#1)+/d  0.2cm/;?(#1)+/u 0.2cm/_{#3}}

\newdir{(}{{}*!<0em,-.14em>-\cir<.14em>{l^r}}
\newdir{ (}{{}*!/-5pt/\dir{(}}
\newdir{ >}{{}*!/-5pt/\dir{>}}

\theoremstyle{definition}

\swapnumbers
\theoremstyle{plain}

\newtheorem{Prop}[subsection]{Proposition}

\numberwithin{equation}{section}

\theoremstyle{definition}

\newtheorem{Rk}[subsection]{Remark}

\newcommand{\Lan}{\mathrm{Lan}}

\begin{document}
\leftmargini=2em \title{Lawvere theories, finitary monads and
  Cauchy-completion} 
\author{Richard Garner} \address{Department of Computing, Macquarie
  University, NSW 2109, Australia} \email{richard.garner@mq.edu.au}

\date{\today}

\thanks{This work was supported by the Australian Research Council's
  \emph{Discovery Projects} scheme, grant number DP110102360.}

\begin{abstract}
  We consider the equivalence of Lawvere theories and finitary monads
  on $\cat{Set}$ from the perspective of
  $\cat{End}_f(\cat{Set})$-enriched category theory, where
  $\cat{End}_f(\cat{Set})$ is the category of finitary endofunctors of
  $\cat{Set}$. We identify finitary monads with one-object
  $\cat{End}_f(\cat{Set})$-categories, and ordinary categories
  admitting finite powers (i.e., $n$-fold products of each object with
  itself) with $\cat{End}_f(\cat{Set})$-categories admitting a certain
  class $\Phi$ of absolute colimits; we then show that, from this
  perspective, the passage from a finitary monad to the associated
  Lawvere theory is given by completion under $\Phi$-colimits. We also
  account for other phenomena from the enriched viewpoint: the
  equivalence of the algebras for a finitary monad with the models of the
  corresponding Lawvere theory; the functorial semantics in arbitrary
  categories with finite powers; and the existence of left
  adjoints to algebraic functors.

\end{abstract}

\maketitle

\section{Introduction}
At the heart of universal algebra is the notion of \emph{equational
  theory}---a first-order theory whose axioms are exclusively of the
form $(\forall x_1) \dots (\forall x_n) (\sigma = \tau)$. There is an
elegant category-theoretic treatment of equational theories due to
Lawvere~\cite{Lawvere1963Functorial}: for each equational theory $\T$,
one may define a category $\cat T$ with finite products such that
models of $\T$ correspond to finite-product-preserving functors $\cat
T \to \cat{Set}$. The objects of $\cat T$ are the distinct finite
powers $X^n$ of a fixed object $X$, whilst morphisms $X^n \to X^m$
are $m$-tuples of derived $n$-ary operations of
$\T$;  a category of this form is called a \emph{Lawvere theory}.

A second way of treating equational theories categorically is using
monads: to each equational theory $\T$, we associate a monad $\mathsf
T$ on the category of sets whose algebras are exactly the $\T$-models.
The value of $\mathsf T$ at a set $X$ is given by the set of derived
terms of the theory with free variables from $X$. Since each derived
term involves only finitely many variables, the action of $\mathsf T$
is entirely determined by its behaviour on finite sets; formally, this
says that the monad $\mathsf T$ is \emph{finitary}, in the sense that
its underlying endofunctor preserves filtered colimits.

The formulations in terms of finitary monads and Lawvere theories are
equivalent; more precisely, we have an equivalence of categories
\begin{equation}\label{eq:1}
   \cat{Mnd}_f(\cat{Set}) \simeq \cat{Law}
\end{equation}
which commutes (to within pseudonatural equivalence) with the functors
sending a finitary monad to its category of algebras, and a Lawvere
theory to its category of models (a \emph{model} of a Lawvere theory
$\cat T$ being, as above, a finite-product-preserving functor $\cat T
\to \cat{Set}$). This equivalence was essentially established by
Linton in~\cite{Linton1966Some}; for modern treatments the reader
could consult, for instance,~\cite{Barr1985Toposes}
or~\cite{Power1999Enriched}; for a historical overview,
see~\cite{Hyland2007The-Category}.

The equivalence~\eqref{eq:1} has been extended in many
directions~\cite{Mac-Lane1965Categorical,Power1999Enriched,Nishizawa2009Lawvere,Lack2011Notions,Berger2012Monads}
to deal with other notions of algebraic structure: for example, ones with different
kinds of arities for the operations, with different rules for handling
variable contexts, or with different objects than sets bearing the
structure. Yet as natural as these extensions are, they do not offer a
compelling explanation as to why the monad--theory
correspondence should exist in the first place. 

This article will attempt such an explanation, making use of a
seemingly unrelated insight of Lawvere: his description
in~\cite{Lawvere1973Metric} of metric spaces as enriched categories in
the sense of~\cite{Kelly1982Basic}. His treatment emphasises
particularly the process of completing an enriched category under
\emph{absolute} colimits, those colimit-types which are preserved by
any functor. Applied to a metric space, seen as an enriched category,
this completion yields the classical Cauchy completion, and the name
\emph{Cauchy-completion} has subsequently come to refer to the
completion of any kind of enriched category under absolute colimits. A
notable application of these ideas is~\cite{Walters1982Sheaves}, which
identifies sheaves on a given site with certain Cauchy-complete
categories enriched over an associated bicategory $\B$. Our
application will use Cauchy-completion over a suitable
enrichment base to explain the monad--theory correspondence.


In more detail, we will consider categories enriched in
$\cat{End}_f(\cat{Set})$, the category of finitary endofunctors of
$\cat{Set}$ with its compositional monoidal structure.  Amongst the
totality of such enriched categories, we find:
\begin{enumerate}
\item Every finitary monad $\mathsf T$ on the category of sets; and
\item Every ordinary category with finite powers, so in particular:
  \begin{enumerate}
  \item Every Lawvere theory $\cat T$;
  \item The category $\cat{Set}$ of sets.
  \end{enumerate}
\end{enumerate}

On the one hand, finitary monads as in (1) are precisely monoids in
$\cat{End}_f(\cat{Set})$, thus, one-object
$\cat{End}_f(\cat{Set})$-categories; on the other, we will identify
ordinary categories as in (2) with the
$\cat{End}_f(\cat{Set})$-categories admitting a certain class $\Phi$
of absolute colimits.  The $\cat{End}_f(\cat{Set})$-categories of the
form (2)---which we call \emph{representable}---are reflective amongst
all $\cat{End}_f(\cat{Set})$-categories, with reflector given by
Cauchy-completion with respect to the class $\Phi$; and the key to our
reconstruction of the equivalence~\eqref{eq:1} is that this
Cauchy-completion applied to a finitary monad $\mathsf T$ yields
precisely the associated Lawvere theory $\cat T$.

This perspective also explains the interaction of the
equivalence~\eqref{eq:1} with models.  On viewing a finitary monad
$\mathsf T$ or a Lawvere theory $\cat T$ as an
$\cat{End}_f(\cat{Set})$-category, we find that the categories
$\cat{Alg}(\mathsf T)$ and $\cat{Mod}(\cat T)$ of algebras or models
are the respective categories of $\cat{End}_f(\cat{Set})$-functors
from $\mathsf T$ or $\cat T$ into $\cat{Set}$.  Since the
$\cat{End}_f(\cat{Set})$-category of sets is representable, the
universal property of the representable reflection asserts the
equivalence
\[ \cat{Alg}(\mathsf T) = \cat{End}_f(\cat{Set})\text-\cat{CAT}(\mathsf T, \cat{Set}) \simeq
\cat{End}_f(\cat{Set})\text-\cat{CAT}(\cat T, \cat{Set}) =
\cat{Mod}(\cat T)
\]
of the algebras of the monad with the models of the theory.

We will account for two further phenomena from the
enriched-categorical perspective. The first is the possibility of
taking models in categories other than $\cat{Set}$. This is most
easily understood in the formulation using Lawvere theories: a model
of a Lawvere theory $\cat T$ can be defined in any category $\C$ with
finite powers as a finite-power-preserving functor $\cat T \to \C$. On
the face of it, it is less clear how to take algebras in $\C$ of a
finitary monad $\mathsf T$ in a way that is functorial in $\mathsf T$
and $\C$. This is where the enriched perspective is superior: both
models of a Lawvere theory and algebras of a finitary monad in $\C$
may be defined with equal simplicity as
$\cat{End}_f(\cat{Set})$-enriched functors from the theory or the
monad into $\C$, seen as an $\cat{End}_f(\cat{Set})$-enriched category.

The second further point we consider is the construction of left
adjoints to algebraic functors. Taking the Lawvere theory perspective,
an \emph{algebraic functor} is a functor $\cat{Mod}(\cat T, \C) \to
\cat{Mod}(\cat S, \C)$ induced by composition with a map $\cat S \to
\cat T$ of Lawvere theories. It is known that such functors have left
adjoints under rather general circumstances
(see~\cite{Kelly1980A-unified}, for example);
we will consider when they may be constructed by
$\cat{End}_f(\cat{Set})$-enriched left Kan extension. It turns out
that this is the case just when unenriched left Kan extensions along
the ordinary functor $\cat S \to \cat T$ exist and distribute
appropriately over finite powers.

In this article, we have only considered the classical correspondence
between finitary monads and Lawvere theories; but each of the
generalisations of the monad--theory correspondence listed above
should also arise in this manner on replacing $\cat{End}_f(\cat{Set})$
by some other appropriate monoidal category or bicategory of
endofunctors; in future work with Hyland, we will study generalised
monad--theory correspondences using enrichment over the Kleisli
bicategories of~\cite{Fiore2008The-cartesian,Hyland2013Elements}.

\section{Finitary monads and their algebras via enriched categories}
\label{sec:finit-monads-their}
\subsection{Finitary monads as $[\cat F, \cat{Set}]$-categories}
In this section, we describe finitary monads on $\cat{Set}$ and their
algebras from the perspective of enriched category theory. As in the
introduction, our base for enrichment will be the category of finitary
endofunctors of $\cat{Set}$; however,
following~\cite[Section~4]{Kelly1993Adjunctions}, we will find it
convenient to work not with $\cat{End}_f(\cat{Set})$ itself, but with
an equivalent and more elementary category.

Let $\cat F$ denote the full subcategory of $\cat{Set}$ spanned by the
finite cardinals. $\cat{F}$ is in fact the free category with finite
colimits on $1$, and so by~\cite[Proposition~5.41]{Kelly1982Basic} the
inclusion $I \colon \cat{F} \to \cat{Set}$ exhibits $\cat{Set}$ as the
free completion of $\cat{F}$ under filtered colimits.  Restriction and
left Kan extension along $I$ thus exhibits $\cat{End}_f(\cat{Set})$ as
equivalent to the functor category $[\cat F, \cat{Set}]$.  The
compositional monoidal structure of $\cat{End}_f(\cat{Set})$
transports across the equivalence to yield a monoidal structure on
$[\cat F, \cat{Set}]$ whose unit object is the inclusion $I$, and
whose tensor product is defined as on the left in:
\[
(A \otimes B)(n) = \textstyle\int^{m \in \cat F} Am \times (Bn)^m
\qquad \quad [B,C](m) = \textstyle\int_{n \in \cat F} [(Bn)^m,
Cn]\rlap{ .}
\]
We record for future use that this monoidal structure is non-symmetric
and \emph{right closed}, meaning that each $(\thg) \otimes B \colon
[\cat F, \cat{Set}] \to [\cat F, \cat{Set}]$ admits a right adjoint
$[B, \thg]$, defined as on the right above. The category of monoids in
$[\cat F, \cat{Set}]$ is, of course, equivalent to the category of
monoids in $\cat{End}_f(\cat{Set})$ and so to the category of finitary
monads on $\cat{Set}$, and so we have:
\begin{Prop}\label{prop:mndf}
  The category $\cat{Mnd}_f(\cat{Set})$ of finitary monads on
  $\cat{Set}$ is equivalent to the category of one-object $[\cat F,
  \cat{Set}]$-categories.
\end{Prop}

\subsection{Monad algebras as $[\cat F, \cat{Set}]$-functors}\label{sec:monadalgfset}
We now describe algebras for finitary monads and the maps
between them in terms of $[\cat F, \cat{Set}]$-enriched functors and
transformations.  For this, we use an analysis which appears in
can be traced back to~\cite[Section~3]{Kelly1974Coherence}; it is
based on certain general considerations
concerning monoidal actions, which may be found, for example, in~\cite{Janelidze2001A-note}.

Suppose that $\V$ is a monoidal category. By a \emph{monoidal action}
of $\V$ on a category $\W$, we mean a functor $\diamond \colon \V
\times \W \to \W$ together with a strong monoidal structure on its
transpose $\V \to [\W, \W]$ (viewing $[\W, \W]$ as strict monoidal
under composition); to give this strong monoidal structure is equally
to give natural isomorphisms $I \diamond X \cong X$ and $(A \otimes
A') \diamond X \cong A \diamond (A' \diamond X)$ satisfying the
evident coherence laws. The action is said to be \emph{right closed}
if each $(\thg) \diamond X \colon \V \to \W$ admits a right adjoint
$\spn{X, \thg} \colon \W \to \V$, with counit components $\varepsilon \colon
\spn{X, Y} \diamond X \to Y$, say.
In these circumstances, the category $\W$ acquires a $\V$-enrichment,
with hom-objects the $\spn{X,Y}$'s and identities and composition $I
\to \spn{X,X}$ and $ \spn{Y,Z} \otimes \spn{X,Y} \to \spn{X,Z}$
obtained by transposing the respective morphisms $I \diamond X \cong
X$ and
\begin{equation*}
  (\spn{Y,Z} \otimes \spn{X,Y}) \diamond X \xrightarrow{\cong}
  \spn{Y,Z} \diamond (\spn{X,Y} \diamond X) \xrightarrow{1 \diamond
    \varepsilon} \spn{Y,Z} \diamond Y
  \xrightarrow{\varepsilon} Z
\end{equation*}
of $\W$ under the right closure adjunctions. Observe that, as
in~\cite[Lemma~2.1]{Janelidze2001A-note}, this structure makes
$\W$ into a \emph{tensored} $\V$-category in the sense of
Section~\ref{sec:tensors} below: the tensor of $X \in \W$ by $A \in
\V$ is given by $A \diamond X$.

Suppose now that $A$ is a monoid in $\V$; its image under the strong
monoidal functor $\V \to [\W, \W]$ is then a monoid in $[\W, \W]$,
hence a monad $A \diamond (\thg)$ on $\W$.
\begin{Prop}\label{prop:actions}
  Given a right closed monoidal action $\diamond \colon \V \times \W
  \to \W$ and a monoid $A \in \V$, we have, on viewing $A$ as
  one-object $\V$-category $\Sigma A$ and equipping $\W$ with the $\V$-enrichment
  derived from the action, an isomorphism of categories $\V\text-\cat{CAT}(\Sigma A, \W)
  \cong (A \diamond \thg)\text-\cat{Alg}$.
\end{Prop}
\begin{proof}
  To give a $\V$-functor $\Sigma A \to \W$ is to give an object $X \in \W$
  and a map $x \colon A \to \spn{X,X}$ in $\V$ making the diagrams
  \[
  \cd[@!C@C-1em]{
    & I \ar[dl]_j \ar[dr]^j \\
    A \ar[rr]_x & & \spn{X,X}} \qquad \text{and} \qquad \cd[@C+0.5em]{
    A \otimes A \ar[r]^-{x \otimes x} \ar[d]_{m} &
    \spn{X,X} \otimes \spn{X,X} \ar[d]^{m}  \\
    A \ar[r]_x & \spn{X,X} }\] commute. Transposing under adjunction,
  this is equally to give an object $X \in \W$ and a map $A \diamond X
  \to X$ satisfying the two axioms for an $A \diamond
  (\thg)$-algebra. To give a $\V$-natural transformation $F
  \Rightarrow G \colon \Sigma A \to \W$ is to
  give a map $\phi \colon I \to \spn{X,Y}$ such that the diagram
  \[
  \cd[@C+1em]{
    A \ar[r]^{y} \ar[d]_{x} & \spn{Y, Y} \ar[r]^-{\spn{Y,Y}\otimes\phi} & \spn{Y,Y} \otimes \spn{X,Y} \ar[d]^{m} \\
    \spn{X, X} \ar[r]_-{\phi\otimes \spn{X,X}} &
    \spn{X,Y}\otimes\spn{X,X} \ar[r]_-{m} & \spn{X, Y} }
  \]
  commutes; which, transposing under adjunction and using the
  coherence constraint $I \diamond X \cong X$, is equally to give a
  map $X \to Y$ commuting with the $A \diamond (\thg)$-actions.  This
  gives the isomorphism $\V\text-\cat{CAT}(\Sigma A, \W) \cong (A \diamond
  \thg)\text-\cat{Alg}$; naturality in $A$ is easily verified.
\end{proof}

We now apply the preceding generalities to the case $\V = [\cat F,
\cat{Set}]$. Transporting the evident monoidal action of
$\cat{End}_f(\cat{Set})$ on $\cat{Set}$ across the equivalence
$\cat{End}_f(\cat{Set}) \simeq [\cat F, \cat{Set}]$ yields a monoidal
action $\diamond \colon [\cat F, \cat{Set}] \times \cat{Set} \to
\cat{Set}$, given as on the left in:
\[A \diamond X = \textstyle\int^{n \in \cat F} An \times X^n \qquad
\quad \spn{X,Y}(n) = \cat{Set}(X^n, Y)\rlap{ .}
\] This action is right closed, with the right adjoints $\spn{X,
  \thg}$ being defined as on the right above. We thus obtain a
canonical enrichment of the category of sets to an $[\cat F,
\cat{Set}]$-category $\Ss$ with hom-objects given by $\Ss(X,Y) =
\cat{Set}(X^{(\thg)}, Y)$; and by the preceding result, we conclude that:
\begin{Prop}\label{prop:monadfset}
  The embedding of finitary monads into one-object $[\cat F,
  \cat{Set}]$-categories fits into a pseudocommuting triangle of
  functors
  \begin{equation*}
    \cd{
      \cat{Mnd}_f(\cat{Set}) \ar[dr]_{(\thg)\text-\cat{Alg}} \ar[rr] &
      \ar@{}[d]|{\textstyle\simeq} & [\cat F, \cat{Set}]\text-\cat{CAT}\rlap{ .} \ar[dl]^{\qquad [\cat F, \cat{Set}]\text-\cat{CAT}(\thg,\,\, \Ss)}\\ & \cat{CAT}
    }
  \end{equation*} 
\end{Prop}

\section{Representable $[\cat F, \cat{Set}]$-categories}
We have now described finitary monads and their algebras in terms of
$[\cat F, \cat{Set}]$-category theory, and in the next section, we
will do the same for Lawvere theories and their models. In the current
section, we set up the theory necessary to do this; as
anticipated in the introduction, this will involve showing that
ordinary categories admitting finite powers (i.e., $n$-fold products
$A \times \dots \times A$ of each object with itself) may identified
with those $[\cat F, \cat{Set}]$-categories admitting a certain class
of absolute colimits.

\subsection{$[\cat F, \cat{Set}]$-categories, functors and transformations}
We begin by unfolding the basic notions of $[\cat F,
\cat{Set}]$-category theory. An $[\cat F, \cat{Set}]$-category $\M$ is
given by the following data:
\begin{enumerate}[(i)]
\item A set of objects $\ob \M$;
\item For all $X, Y \in \ob \M$ and $n \in \cat F$, a homset $\M_n(X,
  Y)$;
\item For all $X, Y \in \ob \M$ and $\phi \colon n \to m \in \cat F$,
  functorial reindexing operations
  \[\phi_\ast \colon \M_n(X, Y) \to \M_m(X, Y)\rlap{ ;}\]
\item For all $X \in \ob \M$, an identity map $1_X \in \M_1(X,
  X)$; and
\item For all $X,Y,Z\in \ob \M$, composition operations, natural in $n$ and
  $m$:
  \begin{equation*}\label{eq:comp}
    \begin{aligned}
      \M_m(Y, Z) \times \M_n(X, Y)^m & \to \M_n(X,Z)\\
      (g, f_1, \dots, f_m) & \mapsto g \circ (f_1, \dots, f_m)\rlap{
        ,}
    \end{aligned}
  \end{equation*}
\end{enumerate}
obeying the following axioms, wherein we write $\pi_1, \dots, \pi_n
\in \M_n(X,X)$ for the images of the element $1_A \in \M_1(X,X)$ under
the $n$ distinct maps $1 \to n$ in $\cat F$:
\begin{enumerate}[(i)]\addtocounter{enumi}{5}
\item $\pi_i \circ (f_1, \dots, f_n) = f_i$;
\item $g \circ (\pi_1, \pi_2, \dots, \pi_n) = g$;
\item $h \circ (g_1 \circ (f_{1}, \dots, f_{k}), \dots, g_j \circ
  (f_{1}, \dots, f_{k})) = (h \circ (g_1, \dots, g_j)) \circ (f_{1},
  \dots, f_{k})$.
\end{enumerate}

An $[\cat F, \cat{Set}]$-functor $F \colon \M \to \N$ is given by an
assignation on objects and assignations on homsets $\M_n(X,Y) \to
\N_n(FX,FY)$ which are natural in $n$ and preserve composition and
identities, whilst an $[\cat F, \cat{Set}]$-transformation $\alpha
\colon F \Rightarrow G$ is given by elements $\alpha_X \in
\N_1(FX,GX)$ such that
\begin{equation}
  \alpha_Y \circ Ff = Gf \circ (\alpha_X \circ \pi_1, \dots, \alpha_X
  \circ \pi_n)\label{eq:fnat}
\end{equation}
for all $f \in \M_n(X,Y)$.  Every $[\cat F, \cat{Set}]$-category $\M$
has an underlying ordinary category $V\M$ with objects those of $\M$
and homsets $V\M(X,Y) = \M_1(X,Y)$; with the evident extension to $1$-
and $2$-cells, we obtain a forgetful $2$-functor $V \colon [\cat F,
\cat{Set}]\text-\cat{CAT} \to \cat{CAT}$.

\begin{Rk}\label{rk:clones}
  Axioms (v) and (vii) in the definition of $[\cat F,
  \cat{Set}]$-category force the maps in (iii) to be given by
  $\phi_\ast(g) = g \circ (\pi_{\phi(1)}, \dots, \pi_{\phi(n)})$; and
  in fact, this leads to an alternative axiomatisation of $[\cat F,
  \cat{Set}]$-categories.  Suppose we are given:
  \begin{enumerate}[(i')]
  \item[(i)] A set of objects $\ob \M$;
  \item[(ii)] For all $X, Y \in \ob \M$ and $n \in \cat F$, a homset
    $\M_n(X, Y)$;
  \item[(iv')] For all $X \in \ob \M$ and $n \in \cat F$, projection
    maps $\pi_1, \dots, \pi_n \in \M_n(X, X)$;
  \item[(v')] For all $X,Y,Z\in \ob \M$, composition operations:
    \begin{equation*}
      \begin{aligned}
        \M_m(Y, Z) \times \M_n(X, Y)^m & \to \M_n(X,Z)\\
        (g, f_1, \dots, f_m) & \mapsto g \circ (f_1, \dots, f_m)\rlap{
          ,}
      \end{aligned}
    \end{equation*}
  \end{enumerate}
  such that axioms (vi)--(viii) are verified. Then we define
  identities as in (iv) by $1_A = \pi_1 \in \M_1(A,A)$, and functorial
  reindexing maps as in (iii) by the above formula; on doing so, (v') becomes
  natural in $n$ and $m$, so yielding~(v).  This alternative
  axiomatisation is a many-object version of the universal
  algebraists' notion of \emph{abstract clone}~\cite{Hall1958Some}. In
  terms of this axiomatisation, an $[\cat F, \cat{Set}]$-functor $\M
  \to \N$ is given by assignations on objects and on homs which
  preserve composition and the projection maps.
\end{Rk}

\begin{Rk}\label{rk:1}
  An $[\cat F, \cat{Set}]$-category also admits a \emph{linear
    composition} operation
  \begin{equation}\label{eq:linear}\begin{aligned}
      \M_m(Y,Z) \times \textstyle \prod_{i = 1}^m \M_{n_i}(X, Y) & \to \M_{\Sigma_i n_i}(X,Z) \\
      (g, f_1, \dots, f_n) & \mapsto g \otimes (f_1, \dots, f_n)
    \end{aligned}
  \end{equation}
  given by $g \otimes (f_1, \dots, f_n) = g \circ
  ((\iota_1)_\ast(f_1), \dots, (\iota_n)_\ast(f_n))$ with $(\iota_j
  \colon n_j \to \Sigma_i n_i)_{j = 1}^m$ the coproduct
  injections. This composition, together with the identity elements in
  (iv) and reindexing maps in (iii), makes $\M$ into a cartesian
  multicategory (called a~\emph{Gentzen multicategory}
  in~\cite{Lambek1989Multicategories}), where we take the set of
  multimaps $X_1, \dots, X_n \to Y$ to be empty unless $X_1 = \dots =
  X_n = X$, in which case we take it to be $\M_n(X,Y)$. This in fact
  gives a further alternative axiomatisation\footnote{The correspondence
    between this axiomatisation and the original one corresponds to
    the correspondence between the ``multiplicative'' and ``additive''
    treatment of contexts in the classical sequent calculus. The basic
    calculation underlying these correspondences is that, for
    cartesian monoidal $\C$, the convolution monoidal structure on
    $[\C^\op, \cat{Set}]$ is again cartesian monoidal.} of $[\cat F,
  \cat{Set}]$-categories: they are precisely the cartesian
  multicategories in which every multimap $X_1, \dots, X_n \to Y$ has
  $X_1 = \dots = X_n$; the key point is that composition (v) is
  definable in terms of~\eqref{eq:linear} and (iii) as:
  \[
  g \circ (f_1, \dots, f_m) = (\pi_1)_\ast(g \otimes (f_1, \dots,
  f_n)) \quad \text{(with $\pi_1 \colon n \times m \to n$ the
    projection).}
  \]
\end{Rk}

\subsection{Tensors in $[\cat F, \cat{Set}]$-categories}\label{sec:tensors}
The relevant colimits for enriched category theory are the
\emph{weighted} (there called \emph{indexed}) limits
of~\cite[Chapter~3]{Kelly1982Basic}. For the moment, we shall need
only the following case of the general notion. Given $\V$ a
right-closed\footnote{Actually,~\cite{Kelly1982Basic} assumes a
  \emph{symmetric} monoidal closed base $\V$, but the definition
  extends without fuss to the non-symmetric, right-closed case.}
monoidal category and $\C$ a $\V$-category, a \emph{tensor} of $X \in
\C$ by $A \in \V$ is an object $Z$ of $\C$ and map $i \colon A \to
\C(X,Z)$ in $\V$ such that for all $Y \in \C$, the composite
\begin{equation}
  \C(Z, Y) \xrightarrow{\C(X,\thg)} [\C(X,Z), \C(X,Y)] \xrightarrow{[i, 1]} [A, \C(X, Y)]\label{eq:rep1}
\end{equation}
is invertible in $\V$; we may sometimes write $Z$ as $A \otimes X$, or
say that \emph{$i$ exhibits $Z$ as $A \otimes X$}.  Taking now $\V =
[\cat F, \cat{Set}]$ and $A = y_n = \cat F(n,\thg)$, we see that, for
an $[\cat F, \cat{Set}]$-category $\M$ and an object $X \in \M$, a
tensor of $X$ by $y_n$ is given by an object $Z$ and map $i \colon y_n
\to \C(X,Z)$---which by the Yoneda lemma is equally an element $i \in
\C_n(X,Z)$---such that for all $Y \in \C$, the map~\eqref{eq:rep1} is
invertible. Unfolding the definitions, this says that for all $Y \in
\C$ and $k \in \cat F$, the function
\begin{equation}\label{eq:cotens-prop}
  \begin{aligned}
    \C_k(Z,Y) &\to \C_{k n}(X, Y)\\
    g & \mapsto g \otimes (i, \dots, i)
  \end{aligned}
\end{equation}
is invertible; here we use the linear composition operation
of~\eqref{eq:linear}.

\begin{Prop}\label{prop:power-copower-fset}
  Let $\M$ be an $[\cat F, \cat{Set}]$-category. For all $X, Z \in
  \M$, the following data are equivalent:
  \begin{enumerate}[(a)]
  \item An element $i \in \M_n(X,Z)$ exhibiting $Z$ as $y_n \otimes
    X$;
  \item Elements $p_1, \dots, p_n \in \M_1(Z,X)$ exhibiting $Z$ as the
    (enriched) power $X^n$;
  \item Elements $i$ and $p_1, \dots, p_n$ as above such that
    \begin{equation}\label{eq:copower-char}
      i \circ (p_1, \dots, p_n) = 1_Z 
      \qquad \text{and} \qquad
      p_k \circ i = \pi_k \text{ for all $1 \leqslant k \leqslant n$.}
    \end{equation}
  \end{enumerate}
  It follows that, in an $[\cat F, \cat{Set}]$-category, tensors by
  representables $y_n$ are \emph{absolute colimits} in the sense of
  being preserved by any $[\cat F, \cat{Set}]$-functor.
\end{Prop}
The universal property asserted in (b) of the maps $p_1, \dots, p_n$
is that, for all $Y \in \C$ and $k \in \cat F$, the map of homsets
$\C_k(Y,Z) \to \C_k(Y, X)^n$ given by postcomposition with $p_1,
\dots, p_n$ is invertible. In particular, this implies that $Z$ is the
power $X^n$ in the underlying ordinary category $V\M$.
\begin{proof}
  Given $i$ as in (a), we define $p_1, \dots, p_n$ as in (c) by the
  universal property of the tensor applied to the maps $\pi_1, \dots,
  \pi_n \in \M_n(X,X)$; then the right-hand equalities
  in~\eqref{eq:copower-char} are immediate, and the left-hand one
  follows on precomposing with $i$ and applying the universal
  property. Conversely, given (c), we obtain the inverse
  to~\eqref{eq:cotens-prop} required for (a) by sending $h \in
  \C_{kn}(X,Y)$ to the composite $h \circ (p_1 \circ \pi_1, \dots, p_1
  \circ \pi_k, \dots, p_n \circ \pi_1, \dots, p_n \circ \pi_k) \in
  \C_k(Z,Y)$.

  On the other hand, given $p_1, \dots, p_n$ as in (b), we define $i$
  as in (c) by the universal property of the power applied to the
  family $(\pi_1, \dots, \pi_n) \in \C_n(X,X)^n$; then the right-hand
  equalities in~\eqref{eq:copower-char} are immediate, and the
  left-hand one follows on postcomposing with $p_1, \dots, p_n$ and
  applying the universal property. Conversely, given (c), we obtain an
  inverse $\C_k(Y,X)^n \to \C_k(Y,Z)$ for postcomposition with $p_1,
  \dots, p_n$, as required for (b), by the mapping $(g_1, \dots, g_n)
  \mapsto i \circ (g_1, \dots, g_n)$.

  Finally, since tensors by representables admit the equational
  reformulation in~(c), they are clearly preserved by any $[\cat F,
  \cat{Set}]$-functor.
\end{proof}
\begin{Rk}\label{rk:absolute}
  The above direct proof could also be deduced from the general
  considerations of~\cite{Street1983Absolute} on absolute
  colimits. Applied to the case of $\V$-enriched tensors, the main
  theorem of ibid.\ states that tensors by $A \in \V$ are absolute
  just when $A$ admits a \emph{left dual} $A^o$ in $\V$, meaning that
  we have maps $\eta \colon I \to A \otimes A^o$ and $\varepsilon
  \colon A^o \otimes A \to I$ satisfying the usual triangle identities
  for an adjunction; moreover, tensors by $A$ may then be identified
  with \emph{cotensors} (the dual limit notion) by $A^o$. Specialising
  to the situation at hand, the object $A = y_n$ of $[\cat F,
  \cat{Set}]$ has left dual $A^o = h_n = n \times (\thg)$, since $h_n$
  and $y_n$ correspond to the adjoint finitary endofunctors $(\thg)
  \times n \dashv (\thg)^n$ of $\cat{Set}$; and so we conclude that
  tensors by $y_n$ are absolute and correspond to cotensors by $h_n$.
  Since $h_n$ is isomorphic to the $n$-fold coproduct $I + \dots + I$
  of the unit object, a cotensor of $X$ by $h_n$ is equally well an
  $[\cat F, \cat{Set}]$-enriched power $X^n$, which gives the
  equivalence (a) $\Leftrightarrow$ (b) of
  Proposition~\ref{prop:power-copower-fset}; a more refined analysis
  of the general case also yields the formulation in (c).
\end{Rk}

\subsection{Representable $[\cat F, \cat{Set}]$-categories}
We define an $[\cat F, \cat{Set}]$-category $\M$ to be
\emph{representable} if it admits all tensors by $y_n$'s. The
terminology is motivated by the observation in Remark~\ref{rk:1} that
a $[\cat F, \cat{Set}]$-category can be seen as a particular kind of
multicategory; when seen in this way, our notion of representability
reduces to the standard notion of representability for
multicategories---as given, for example,
in~\cite[Definition~3.3.1]{Leinster2004Higher}.

Let us write $[\cat F, \cat{Set}]\text-\cat{CAT}_\mathrm{rep}$ for the
full sub-$2$-category on the representable $[\cat F,
\cat{Set}]$-categories.  By (a) $\Rightarrow$ (b) in
Proposition~\ref{prop:power-copower-fset}, the underlying category
$V\M$ of a representable $[\cat F, \cat{Set}]$-category admits all
finite powers; and by the absoluteness of tensors by representables,
the underlying functor $VF$ of any $[\cat F, \cat{Set}]$-functor $F
\colon \M \to \N$ between representable $[\cat F,
\cat{Set}]$-categories preserves finite powers. Thus the restriction
of the underlying category $2$-functor to representable $[\cat F,
\cat{Set}]$-categories factors through $\cat{CAT}_\mathrm{fp}$, the
$2$-category of categories with finite powers and
finite-power-preserving functors.

\begin{Prop}\label{prop:fsetfp}
  The $2$-functor $V \colon [\cat F,
  \cat{Set}]\text-\cat{CAT}_\mathrm{rep} \to \cat{CAT}_\mathrm{fp}$ is
  an equivalence of $2$-categories.
\end{Prop}
Bearing in mind the identification of $[\cat F, \cat{Set}]$-categories
with particular cartesian multicategories, this result is essentially
a special case of the equivalence between cartesian multicategories
and categories with finite products
(cf.~\cite{Lambek1989Multicategories}); however, the proof is
short enough to bear repeating here.

\begin{proof}
  We exhibit a pseudoinverse $2$-functor $R \colon
  \cat{CAT}_\mathrm{fp} \to [\cat F,
  \cat{Set}]\text-\cat{CAT}_\mathrm{rep}$. For a category $\C$ with
  finite powers, we take $R\C$ to have the same collection of objects,
  hom-sets $R\C_n(X,Y) = \C(X^n, Y)$, composition given by:
  \begin{align*}
    \C(Y^m, Z) \times \C(X^n, Y)^m &\to \C(X^n, Z)\\
    (g, f_1, \dots, f_m) & \mapsto g \circ \spn{f_1, \dots, f_m}
  \end{align*}
  and projections $\pi_1, \dots, \pi_n \in R\C_n(X,X) = \C(X^n, X)$
  the product projections. Note that $R\C$ is representable, since for
  any $X \in R\C$ and $n \in \cat F$, the elements $1_{X^n} \in
  R\C_n(X, X^n)$ and $\pi_1, \dots, \pi_n \in R\C_1(X^n, X)$ satisfy
  \eqref{eq:copower-char} and so exhibit $X^n$ as a tensor of $X$ by
  $y_n$. Given next a functor $F \colon \C \to \D$ in
  $\cat{CAT}_\mathrm{fp}$, the $[\cat F, \cat{Set}]$-functor $RF$ has
  the same action on objects and action on homs:
  \begin{equation*}
    \C(X^n, Y) \xrightarrow{F} \D(F(X^n), FY) \xrightarrow{\cong} \D((FX)^n, FY)\rlap{
      .}
  \end{equation*}
  Finally, for any $\alpha \colon F \Rightarrow G$ in
  $\cat{CAT}_\mathrm{fp}$ we take $R\alpha$ to be the $[\cat F,
  \cat{Set}]$-transformation with the same components; $[\cat F,
  \cat{Set}]$-naturality is easily verified. It is clear that the
  $2$-functor $R$ so defined satisfies $VR \cong 1$, and it remains to
  show that $RV \cong 1$.

  Given $\M$ a representable $[\cat F, \cat{Set}]$-category, $RV\M$
  has the same objects and hom-sets $(RV\M)_n(X,Y) = \M_1(X^n,Y)$,
  where $X^n$ is a chosen tensor of $X$ by $y_n$. The projection maps
  in $RV\M_n(X,X) = \M_1(X^n, X)$ are the maps $p_1, \dots, p_n$
  exhibiting $X^n$ as the $n$-fold power of $X$; whilst composition is
  given by
  \begin{align*}
    \M_1(Y^m, Z) \times \M_1(X^n, Y)^m &\to \M_1(X^n, Z) \\
    (g, f_1, \dots, f_m) & \mapsto g \circ i \circ (f_1, \dots, f_n)
  \end{align*}
  where $i \in \M_m(Y, Y^m)$ exhibits $Y^m$ as the tensor of $Y$ by
  $y_m$.  We define an identity-on-objects $[\cat F,
  \cat{Set}]$-functor $RV\M \to \M$ with action on homs given by
  \begin{align*}
    \M_1(X^n, Y) & \to \M_n(X,Y) \\
    g & \mapsto g \circ i
  \end{align*}
  where $i \in \M_n(X, X^n)$ exhibits $X^n$ as the tensor of $X$ by
  $y_n$. The universal property implies that the actions on homs are
  invertible, and it is immediate from the definitions that
  composition and projections are preserved. We thus have an
  isomorphism $RV \M \to \M$; the naturality of these isomorphisms in
  $\M$ is now easily verified.
\end{proof}

\section{Lawvere theories and their models}\label{sec:lawv}
\subsection{Lawvere theories as $[\cat F, \cat{Set}]$-categories}
As in the introduction, a \emph{Lawvere theory} is a category $\cat T$
with finite products whose objects are the distinct finite powers
$X^n$ of a distinguished object $X$. A morphism of Lawvere theories is
a functor $\cat T \to \cat T'$ strictly preserving finite products and
the distinguished object. Note that the condition on the objects of a
Lawvere theory means that we can replace ``finite products''
everywhere in the above by ``finite powers''; and on doing so, the
definitions immediately translate via Proposition~\ref{prop:fsetfp}
into the language of $[\cat F, \cat{Set}]$-category theory.  We call a
representable $[\cat F, \cat{Set}]$-category $\M$ \emph{Lawvere} if
its objects are the tensors $y_n \otimes X$ by distinct representables
of a distinguished object $X$; a functor between two such categories
is called \emph{Lawvere} if it strictly preserves the distinguished
object and its chosen tensors. It is now immediate from
Proposition~\ref{prop:fsetfp} that:
\begin{Prop}\label{prop:law}
  The category  of Lawvere theories is equivalent to the
  category of Lawvere $[\cat F, \cat{Set}]$-categories
  and Lawvere functors.
\end{Prop}

\subsection{Models of Lawvere theories as $[\cat F, \cat{Set}]$-functors}
A \emph{model} of a Lawvere theory $\cat T$ is a
finite-product-preserving functor $\cat T \to \cat{Set}$; as before,
the restriction imposed on the objects of $\cat T$  means that this is
equivalently a finite-power-preserving functor $\cat T \to \cat{Set}$,
which by Proposition~\ref{prop:fsetfp}, is equally a $[\cat F,
\cat{Set}]$-functor $R\cat T \to R(\cat{Set})$. Note that
$R(\cat{Set})$ is precisely the $[\cat F, \cat{Set}]$-category $\Ss$
defined before Proposition~\ref{prop:monadfset}, and so we have:
\begin{Prop}\label{prop:lawvfset}
  The functor which views a Lawvere theory as an $[\cat
  F, \cat{Set}]$-category fits into a pseudocommuting triangle
  \begin{equation*}
    \cd[@!C@C-3em]{
      \cat{Law} \ar[dr]_{(\thg)\text-\cat{Mod}} \ar[rr] &
      \ar@{}[d]|{\textstyle\simeq} & [\cat F, \cat{Set}]\text-\cat{CAT}_\mathrm{}\rlap{ .} \ar[dl]^{\qquad [\cat F, \cat{Set}]\text-\cat{CAT}(\thg,\,\, \Ss)}\\ & \cat{CAT}
    }
  \end{equation*} 
\end{Prop}

\section{The equivalence of finitary monads and Lawvere theories}
\subsection{The representable completion}
Having described both finitary monads on $\cat{Set}$ and Lawvere
theories in terms of $[\cat F, \cat{Set}]$-category theory, we now
describe their equivalence in the same terms. The following result is
the key to doing so.
\begin{Prop}\label{prop:rep-reflect}
  The inclusion $2$-functor $[\cat F,
  \cat{Set}]\text-\cat{CAT}_\mathrm{rep} \to [\cat F,
  \cat{Set}]\text-\cat{CAT}$ admits a left biadjoint $L$.
\end{Prop}
As with Proposition~\ref{prop:fsetfp}, this result is an essentially
standard one about representability in multicategories; see, for
instance, \cite[Section~7]{Hermida2000Representable}. The point is not
that the result is new, but rather that the proof we give  involves only
 standard enriched-categorical notions.
\begin{proof}
  An $[\cat F, \cat{Set}]$-category is representable just when it
  admits certain absolute colimits, namely tensors by representables;
  so $L$ must be given by completion under these
  colimits. By~\cite[Proposition~5.62]{Kelly1982Basic} and the
  absoluteness of the colimits at issue, the unit $J \colon \M \to
  L\M$ of the biadjunction at $\M$ is characterised by three
  properties: (i) $J$ is fully faithful; (ii) $L\M$ is representable;
  (iii) every object of $L\M$ is a tensor by some $y_n$ of an
  object of $\M$. We may thus obtain $L\M$ by first forming the
  \emph{Cauchy completion} $Q\M$ of $\M$---its completion under all
  absolute colimits, described
  in~\cite[Section~1]{Betti1982Cauchy-completion}---and then taking
  $L\M$ to be the closure of $\M$ in $Q\M$ under tensors by
  representables.

  Since tensors by representables satisfy $y_1 \otimes X \cong X$ and
  $y_n \otimes (y_m \otimes X) \cong (y_n \otimes y_m) \otimes X \cong
  y_{nm} \otimes X$, this closure process converges after one step, and
  so we may as well take the objects of $L\M$ to be of the form
  $X^{(n)}$, representing the tensor of $X \in \M$ by $y_n$. Now from
  the description of $Q\M$ given in~\cite{Betti1982Cauchy-completion},
  the hom-objects of $L\M$ are given by
  \[L\M(X^{(n)}, Y^{(m)}) = Q\M(y_n \otimes X, y_m \otimes Y) = y_m
  \otimes \M(X,Y) \otimes h_n\rlap{ ,}\] where as in
  Remark~\ref{rk:absolute}, $h_n = (\thg) \times n$ is the left dual
  of $y_n$ in $[\cat F, \cat{Set}]$. Identities and composition in
  $L\M$ are obtained from those of $\M$ together with the unit maps $I
  \to y_n \otimes h_n$ (for the identities) and counit maps $h_m
  \otimes y_m \to I$ (for the composition). Spelling this out
  explicitly, we have that
  \[
  L\M_k(X^{(n)}, Y^{(m)}) = \M_{nk}(X,Y)^m
  \]
  with identities given by $(\pi_1, \dots, \pi_n) \in L\M_1(X^{(n)},
  X^{(n)}) = \M_n(X,X)^n$, and composition $\L\M_p(Y^{(m)}, Z^{(k)})
  \times \L\M_q(X^{(n)}, Y^{(m)})^p \to L\M_q(X^{(n)}, Z^{(k)})$
   by
  \begin{align*}
    \M_{mp}(Y,Z)^k \times \M_{nq}(X,Y)^{mp} \to \M_{nq}(X, Z)^k\\
    (f_1, \dots, f_k,\, \vec g) \mapsto (f_1 \circ \vec g, \dots, f_k
    \circ \vec g)\rlap{ .}
  \end{align*}
  The tensor of $X^{(n)}$ by $y_m$ is $X^{(nm)}$, as witnessed by the
  element $i = (\pi_1, \dots, \pi_{nm}) \in L\M_m(X^{(n)}, X^{(nm)}) =
  \M_{nm}(X,X)^{nm}$.  Finally, the reflection map $\M \to L\M$ sends
  $X$ to $X^{(1)}$ and is the identity on homs. \end{proof}
\begin{Rk}
  A priori, the universal property of biadjunction only makes $L$
  pseudofunctorial in $\M$; however, it is easy to see that we may
  make it strictly $2$-functorial, by defining its action $LF \colon
  L\M \to L\N$ on morphisms by $(LF)(X^{(n)}) = (FX)^{(n)}$ and with
  the evident action on homs.
\end{Rk}
We have now developed enough $[\cat F, \cat{Set}]$-category theory to
prove:
\begin{Prop}
  There is an equivalence $\cat{Mnd}_f(\cat{Set}) \simeq \cat{Law}$,
  fitting into a pseudocommuting triangle of functors
  \begin{equation}\label{eq:models}
    \cd[@!C@C-2em]{
      \cat{Mnd}_f(\cat{Set})\ar[dr]_{(\thg)\text-\cat{Alg}} \ar[rr] &
      \ar@{}[d]|{\textstyle\simeq} &\cat{Law}\rlap{ .} \ar[dl]^{\qquad (\thg)\text-\cat{Mod}}\\ & \cat{CAT}
    }
  \end{equation} 
\end{Prop}
\begin{proof}
  To show that $\cat{Mnd}_f(\cat{Set}) \simeq \cat{Law}$, it suffices
  by Propositions~\ref{prop:mndf} and~\ref{prop:law} to exhibit an
  equivalence between the category $\A$ of one-object $[\cat F,
  \cat{Set}]$-categories and the category $\B$ of Lawvere $[\cat F,
  \cat{Set}]$-categories and Lawvere functors.  In one direction,
  there is a functor $\B \to \A$ sending each Lawvere $\M$ to the
  one-object sub-$[\cat F, \cat{Set}]$-category $\M_X$ on the
  distinguished object $X$. In the other, if $\M$ is a $[\cat F,
  \cat{Set}]$-category with unique object $X$, then $L\M$ becomes
  Lawvere when equipped with the distinguished object $X^{(1)}$; and
  so we have a functor $\A \to \B$. If $\M$ has one object, then
  clearly $(L\M)_{X^{(1)}} \cong \M$; so the composite $\A \to \B \to
  \A$ is isomorphic to the identity. On the other hand, if $\M$ is
  Lawvere, then the inclusion $\M_X \to \M$ satisfies conditions
  (i)--(iii) from the proof of Proposition~\ref{prop:rep-reflect}, and
  so exhibits $\M$ as the free representable $[\cat F,
  \cat{Set}]$-category on $\M_X$; thus $L\M_X \simeq \M$. This equivalence is in fact bijective on objects, so
  that $L\M_X \cong \M$ and the composite $\B \to \A \to \B$ is
   isomorphic to the identity, as required.

  Finally, we must show that the triangle~\eqref{eq:models} commutes
  to within pseudonatural equivalence. Consider the diagram
  \begin{equation*}
    \cd[@!C@C-4em]{
      \cat{Mnd}_f(\cat{Set})\ar[d] \ar[rr] \twocong{drr} & &
      \cat{Law} \ar[d] \rlap{ .} \\
      [\cat F, \cat{Set}]\text-\cat{CAT} \ar[rr]^L  \ar[dr]_{[\cat F,
        \cat{Set}]\text-\cat{CAT}(\thg,\,\, \Ss)\qquad}& \ar@{}[d]|{\textstyle\simeq} &
      [\cat F, \cat{Set}]\text-\cat{CAT}_\mathrm{rep}  \ar[dl]^{\qquad [\cat F, \cat{Set}]\text-\cat{CAT}(\thg,\,\, \Ss)} \\ &
      \cat{CAT}
    }
  \end{equation*}

  The top square commutes to within isomorphism by our construction of
  the equivalence $\cat{Mnd}_f(\cat{Set}) \simeq \cat{Law}$; whilst
  the lower triangle commutes to within pseudonatural equivalence
  because $L$ is a bireflector into representable $[\cat F,
  \cat{Set}]$-categories and $\Ss$ is representable. Finally, by
  Propositions~\ref{prop:monadfset} and~\ref{prop:lawvfset}, the
  composites down the left and the right are pseudonaturally
  equivalent to $(\thg)\text-\cat{Alg}$ and $(\thg)\text-\cat{Mod}$
  respectively; whence the result.
\end{proof}

\section{Functorial semantics}
One advantage of Lawvere theories over finitary monads is the relative
ease with which we may consider models in categories other than
$\cat{Set}$. A model of a Lawvere theory $\cat T$ in a category $\C$
with finite powers is simply a finite-power-preserving functor $\cat T
\to \C$, and the formulation makes it apparent that any
finite-power-preserving functor between theories $\cat T \to \cat S$
or semantic domains $\C \to \D$ induces a functor $\cat{Mod}(\cat S,
\C) \to \cat{Mod}(\cat T, \C)$ or $\cat{Mod}(\cat T, \C) \to
\cat{Mod}(\cat T, \D)$ by pre- or postcomposition, respectively.  For
a finitary monad on $\cat{Set}$, by contrast, it requires work to
define algebras in other categories, and further work to verify the
functoriality of such a definition in the monad $\mathsf T$ and
the semantic domain $\C$.

The perspective of $[\cat F, \cat{Set}]$-category theory dissolves
this apparent distinction. We retain the functorial semantics for
Lawvere theories by defining $\cat{Mod}(\mathsf T, \C) = [\cat F,
\cat{Set}]\text-\cat{CAT}(R\cat T, R\C)$, but now have an equally
clear functorial semantics for monads on taking $\cat{Alg}(\mathsf T,
\C) = [\cat F, \cat{Set}]\text-\cat{CAT}(\Sigma \mathsf T,
R\C)$. Moreover, the two kinds of semantics are equivalent: when $\cat
T$ is the Lawvere theory corresponding to $\mathsf T$, we have $R\cat T
\cong L\Sigma \mathsf T$, so that by the universal property of the
representable completion,
\[
\cat{Alg}(\mathsf T, \C) = [\cat F, \cat{Set}]\text-\cat{CAT}(\Sigma\mathsf
T, R\C) \simeq  [\cat F, \cat{Set}]\text-\cat{CAT}(R\cat
T, R\C) = \cat{Mod}(\cat T, \C)
\]
pseudonaturally in the representable $[\cat F, \cat{Set}]$-category
$\C$.


In elementary terms, a $\mathsf T$-algebra $\Sigma \mathsf T \to R \C$
is given by an object $X \in \C$ and a monad morphism $\mathsf T \to
\mathsf{End}(X)$; here, $\mathsf{End}(X) = R\C(X,X)$ is the finitary
monad on $\cat{Set}$ with action on finite sets $n \mapsto
\C(X^n,X)$. A map between $\mathsf T$-algebras is a morphism $f \colon
X \to Y$ of $\C$ making the square of finitary endofunctors
\begin{equation*}
  \cd[@C+1em]{
     T \ar[r] \ar[d] & R\C(X,X) \ar[d]^{R\C(X,f)} \\
     R\C(Y,Y) \ar[r]^-{ R\C(f,X)} &  R\C(X,Y)
  }
\end{equation*}
commute; the functor $\cat{Alg}(\mathsf T, \C) \to \cat{Alg}(\mathsf
S, \C)$ induced by a map of monads $\mathsf S \to \mathsf T$ sends
$\mathsf T \to \mathsf{End}(X)$ to $\mathsf S \to \mathsf T \to
\mathsf{End}(X)$; whilst the functor $\cat{Alg}(\mathsf T, \C) \to
\cat{Alg}(\mathsf T, \D)$ induced by a finite-power-preserving $F
\colon \C \to \D$ sends $\mathsf T \to \mathsf{End}(X)$ to $\mathsf T
\to \mathsf{End}(X) \to \mathsf{End}(FX)$, where $\mathsf{End}(X) \to
\mathsf{End}(FX)$ is the finitary monad map defined at $n$ by $\C(X^n,
X) \to \D(F(X^n),FX) \cong \D((FX)^n, FX)$. Let us make it clear that
these definitions are by no means new\footnote{The construction of
  $\mathsf{End}(X)$ is in~\cite[Section~2]{Street1972The-formal} but
   dates back to Lawvere's thesis~\cite{Lawvere1963Functorial}; the
  analysis in the form just given is essentially
  in~\cite[Section~3]{Kelly1974Coherence}}; the point is that, from
the $[\cat F, \cat{Set}]$-enriched viewpoint, they are essentially
forced upon us.



\section{Left adjoints to algebraic functors}
A functor $\cat{Alg}(\mathsf T, \C) \to \cat{Alg}(\mathsf S, \C)$ or
$\cat{Mod}(\cat T, \C) \to \cat{Mod}(\cat S, \C)$ induced by
precomposition with a map of finitary monads or of Lawvere theories is
called an \emph{algebraic functor}. Under reasonable hypotheses on
$\C$, such functors have left adjoints; in this final section, we consider
the case where such adjoints can be constructed from $[\cat F,
\cat{Set}]$-enriched left Kan extensions in the sense
of~\cite{Kelly1982Basic}.


\subsection{Left Kan extensions}\label{sec:leftkan}
Given $\V$-functors $F \colon \A \to \B$ and $G  \colon \A \to
\D$, the \emph{left Kan extension}\footnote{We follow Kelly in
  reserving the name ``left Kan extension'' for
  what~\cite{Dubuc1970Kan-extensions,Mac-Lane1971Categories} call a
  \emph{pointwise} left Kan extension: one computed at each object by
  a weighted colimit in the codomain category. A left adjoint to an
  algebraic functor may exist without being computed by pointwise Kan
  extension---for example, $\cat{Alg}(\mathsf S, \C)$ and
  $\cat{Alg}(\mathsf T, \C)$ could be locally presentable, and the
  adjoint constructed by the methods
  of~\cite{Kelly1980A-unified}---but that relies on something more
  than intrinsic $[\cat F, \cat{Set}]$-categorical properties of $\C$,
  and so is outside our remit.} of $G$ along $F $ is the
$\V$-functor $\Lan_F  G \colon \B \to \D$ defined by the colimit
formula $(\Lan_F G)(B) = \B(F \thg,B) \star G$.  If
$\Lan_F G$ exists for all $G \colon \A \to \D$, then
by~\cite[Theorem~4.43]{Kelly1982Basic} it provides the values of a
(ordinary) functor
\[
\Lan_F \colon \V\text-\cat{CAT}(\A, \D) \to \V\text-\cat{CAT}(\B,
\D)\rlap{ ,}\] left adjoint to precomposition with $F $. When $\V$ is
$[\cat F, \cat{Set}]$, $\D = R\C$ is a representable $[\cat F,
\cat{Set}]$-category, and $F \colon \Sigma \mathsf S \to \Sigma
\mathsf T$ is the $[\cat F, \cat{Set}]$-functor induced by a map of
finitary monads, we see that $\Lan_F $ must provide a left adjoint for
the algebraic functor $\cat{Alg}(\mathsf T, \C) \to \cat{Alg}(\mathsf
S, \C)$; similarly, when $F$ is the $[\cat F, \cat{Set}]$-functor
 $R\cat S \to R\cat T$ induced by a map of Lawvere theories, $\Lan_{F }$
must provide a left adjoint for the algebraic functor $\cat{Mod}(\cat
T, \C) \to \cat{Mod}(\cat S, \C)$.

\subsection{Relative tensors}\label{sec:relativetensors}
In taking the left Kan extensions yielding left adjoints to algebraic
functors, we will need a more general kind of weighted colimit than
the tensors introduced previously.
The following definitions are special cases of ones
in~\cite{Street1983Enriched}; note that the material
of~\cite{Kelly1982Basic} is not applicable, as it assumes that $\V$ is
biclosed symmetric, whereas we assume only right closure without
symmetry.

Let $\Ss = \Sigma S$ be a one-object $\V$-category; by a \emph{right
  $\Ss$-module}, we mean a right module for the underlying monoid
$S$. If $b \colon B \otimes S \to B$ is a right $\Ss$-module, then so
too is $A \otimes b \colon A \otimes B \otimes S \to A \otimes B$, and
the assignation $(A,b) \mapsto A \otimes b$ underlies a monoidal
action $\V \times \Ss\text-\cat{Mod} \to
\Ss\text-\cat{Mod}$. The action is right closed, with $\spn{B,C}$
being defined as the equaliser of the two maps
\[
[B,C] \xrightarrow{[b,1]} [B \otimes S, C] \quad \text{and} \quad [B,C]
\xrightarrow{\thg \otimes S} [B \otimes S, C \otimes S]
\xrightarrow{[1,c]} [B\otimes S, C]\rlap{ ,}
\]
so that, by the argument of Section~\ref{sec:monadalgfset}, we have an
enrichment of $\Ss\text-\cat{Mod}$ to a tensored $\V$-category; when
seen in this way, we write it as  $\P \Ss$.

Given a $\V$-functor $\Ss \to \C$, comprising an object $X \in \C$ and a
monoid morphism $x \colon S \to \C(X,X)$ which we might think of as a
\emph{left $\Ss$-action on $X$}, we have a lifting of the hom-functor
$\C(X,\thg) \colon \C \to \V$ through the forgetful $\P \Ss \to \V$,
obtained by equipping each $\C(X,Y)$ with the right $\Ss$-action
\[\C(X,Y) \otimes S \xrightarrow{1 \otimes x} \C(X,Y) \otimes \C(X,X)
\xrightarrow{m} \C(X,Y)\rlap{ .}\] Now given $X$ with its left
$\Ss$-action and a right $\Ss$-module $A$, the \emph{relative tensor of
  $X$ by $B$ over $\Ss$} is given by an object $A \otimes_\Ss X \in
\C$ and a map $i \colon A \to \C(X,A \otimes_\Ss X)$ of right
$\Ss$-modules such that, for every $Y \in \C$, the map
\[
\C(A \otimes_\Ss X,Y) \xrightarrow{\C(X,\thg)} \P \Ss(\C(X,A \otimes_\Ss X), \C(X,Y)) \xrightarrow{\P
  \Ss(i, 1)} \P \Ss(A, \C(X,Y))
\]
is invertible in $\V$. 

\subsection{Relative tensors in $[\cat F, \cat{Set}]$-categories}
The key to describing relative tensors in $[\cat F,
\cat{Set}]$-categories is the following result, a standard part of the
folklore on algebraic theories:

\begin{Prop}\label{prop:nerve}
  If $\mathsf S$ is a finitary monad on $\cat{Set}$ and $\cat S$ the
  corresponding Lawvere theory, then $\P \Sigma \mathsf S \cong R[\cat
  S^\op, \cat{Set}]$.
\end{Prop}
\begin{proof}
  By its construction, $\P \Sigma \mathsf S$ is a tensored $[\cat F,
  \cat{Set}]$-category, and so in particular representable; it thus
  suffices to show that its underlying ordinary category is isomorphic
  to $[\cat S^\op, \cat{Set}]$.  Now, for any $A \in [\cat F,
  \cat{Set}]$, to give a map $ A\otimes S \to A$ is to give maps $An
  \times (Sm)^n \to Am$ or equally maps $(Sm)^n \to (Am)^{An}$,
  natural in $n$ and $m$. Since $\cat S(n,m) = (Sn)^m$,
  this is to give a graph morphism $\cat S^\op \to \cat{Set}$ which
  (by naturality in $n$ and $m$) restricts along $\cat F \to \cat
  S^\op$ to give back $A$. Imposing the requirement that $A \otimes S
  \to A$ satisfy the unit and associativity conditions for a right
  module now forces the graph morphism $\cat S^\op \to \cat{Set}$ to
  be a functor, so that, in sum, a right $\Sigma \mathsf S$-module is equally
  a pair of $A \colon \cat F \to \cat{Set}$ together with an extension
  of $A$ through $\cat F \to \cat{S}^\op$; which is equally just a
  functor $\cat S^\op \to \cat{Set}$. Arguing similarly for the
  morphisms, we conclude that the underlying category of $\P \Sigma \mathsf
  S$ is isomorphic to $[\cat S^\op, \cat{Set}]$, as claimed.
\end{proof}

We now characterise relative tensors in a representable $[\cat F,
\cat{Set}]$-category in terms of colimits in the underlying ordinary
category which distributive over finite powers. First we make the
sense of this distributivity precise.  Let $\C$ be a category with
finite powers, and $D \colon \A \to \C$ a functor whose values are
taken in powers $X^n$ of some fixed object $X$ of $\C$. Suppose that
$i \colon D \Rightarrow \Delta Z$ is a colimiting cocone for $D$.  For
each $k \in \mathbb N$, write $D^k \colon \A^k \to \C$ for the functor
$(a_1, \dots, a_k) \mapsto D(a_1) \times \dots \times D(a_k)$ (note
these products will exist by the assumption on $D$), and $i^k \colon
D^k \Rightarrow \Delta(Z^k)$ for the induced cocone with components
$i_{a_1} \times \dots \times i_{a_k}$. If the cocone $i^k$ is
colimiting for each $k$, we say that the colimit of $D$
\emph{distributes over finite powers}.

\begin{Prop}\label{prop:fsetreltensors}
  Let $\C$ be a category with finite powers, and $\mathsf S$ a
  finitary monad on $\cat{Set}$. The relative tensor of $X \colon
  \Sigma \mathsf S \to R\C$ by $A \in \P \Sigma \mathsf S$ exists if and only
  if the composite ordinary functor \vskip-0.7\baselineskip
  \begin{equation}\label{eq:tensordiagram2}
    D = \mathrm{el}\, {\tilde A} \to {\cat S} \xrightarrow{\tilde X} \C
  \end{equation}
  admits a colimit which distributes over finite powers; here, $\cat
  S$ is the Lawvere theory associated to $\mathsf S$, the presheaf
  $\tilde A \in [\cat S^\op, \cat{Set}]$ corresponds to $A \in \P
  \Sigma \mathsf S$ under Proposition~\ref{prop:nerve}, and $\tilde X \colon
  \cat S \to \C$ is the essentially-unique finite-power-preserving
  functor whose restriction $\Sigma \mathsf S \to R\cat S \to R \C$ is
  isomorphic to $X$.
\end{Prop}
The colimit of $D$ is thus the (unenriched) weighted colimit
$\tilde A \star \tilde X$, given in coend notation\footnote{This is
  merely notation: we are only asserting the existence of the colimit
  whose universal property is expressed by this coend, and not that of
  the copowers $\tilde An \cdot \tilde Xn$ constituting it.} by
$\int^{n \in \cat S} \tilde An \cdot \tilde Xn$. The distributivity of
the colimit over finite powers is the requirement that, for all $k \in
\mathbb N$, we have a canonical isomorphism
\[
\textstyle (\int^n \tilde An \cdot \tilde Xn)^k \cong \int^{n_1,
  \dots, n_k} (\tilde An_1 \times \dots \times \tilde An_k)
\cdot \tilde X(n_1 +  \dots + n_k)\rlap{ ,}
\]
in the sense that the evident cocone of maps exhibits the left-hand
side as the colimit on the right.  Note that if $\C$ is a cartesian
closed category, then any colimit in $\C$ distributes over finite
powers; thus if $\C$ is cartesian closed and cocomplete (in
particular, if $\C = \cat{Set}$) then $R\C$
admits all $[\cat F, \cat{Set}]$-enriched relative tensors.
\begin{proof}
  For any $Y \in \C$, we induce as in
  Section~\ref{sec:relativetensors} a right $\Sigma \mathsf S$-module
  structure on the hom-object $R\C(X,Y)$, which under the isomorphism
  of Proposition~\ref{prop:nerve} is easily identified with the
  presheaf $\C(\tilde X, Y) \in [\cat S^\op, \cat{Set}]$. In these
  terms, the the universal property of the tensor $Z = A
  \otimes_\mathsf S X$ mandates isomorphisms
\begin{equation}\label{eq:tensorprop2}
  R\C(Z, Y) \cong R[\cat S^\op, \cat{Set}](\tilde A,\C(\tilde
  X,Y))
\end{equation}
in $[\cat F, \cat{Set}]$, induced by composition with a universal map
$i \colon \tilde A \to \C(\tilde X, Z)$ in $[\cat S^\op, \cat{Set}]$.
To give $i$ is to give functions $\tilde An \to \C(\tilde
Xn, Z)$ natural in $n \in \cat S$, thus a cocone $i \colon D
\Rightarrow \Delta Z$ under~\eqref{eq:tensordiagram2} with vertex
$Z$. Evaluating~\eqref{eq:tensorprop2} at $1 \in \cat F$, we see that
this cocone must be colimiting; evaluating at $k \in \cat F$, we find
that composition with $i^k$ induces isomorphisms $\C(Z^k, Y) \cong
[\cat S^\op, \cat{Set}](\tilde A^k, \C(\tilde X, Y))$. Since $[\cat
S^\op, \cat{Set}]$ is cartesian closed, we calculate that
\begin{equation}\label{eq:calculation}\begin{aligned}
  \tilde A^k = \textstyle (\int^n \tilde An \cdot y_n)^k & \cong
  \textstyle\int^{n_1, \dots, n_k} (\tilde An_1 \times \dots \times
  \tilde An_k)
  \cdot y_{n_1} \times \dots \times y_{n_k}\\
  & \cong \textstyle\int^{n_1, \dots, n_k}(\tilde An_1 \times \dots
  \times \tilde An_k) \cdot y_{n_1 + \dots + n_k}\rlap{ ,}
\end{aligned}
\end{equation}
so that $[\cat S^\op, \cat{Set}](\tilde A^k, \C_1(\tilde X, Y))$ is equally the
set of cocones $D^k \Rightarrow \Delta Y$; the natural isomorphism of
this with $\C(Z^k, Y)$ specified by~\eqref{eq:tensorprop2} now asserts
that $i^k \colon D^k \to \Delta  (Z^k)$ is a colimiting cocone, as
required.
\end{proof}

\subsection{Left adjoints to algebraic functors}
We are now finally in a position to describe when left adjoints to
algebraic functors can be obtained by $[\cat F, \cat{Set}]$-enriched
Kan extension. The construction given in the following result is once
again not new, at least when stated in the form stated in (iii) and
(iv); what is new is its abstract justification via the universal
property in (i) and (ii).

\begin{Prop}\label{prop:freetalg}
  Let $F \colon \mathsf S \to \mathsf T$ be a map of finitary monads
  on $\cat{Set}$ and $G \colon \cat S \to \cat T$ the associated map
  of Lawvere theories, inducing the left-hand square of $[\cat F,
  \cat{Set}]$-functors in:
  \begin{equation*}
    \cd{
      \Sigma \mathsf S \ar[d]_{\Sigma F} \ar[r]^{\eta} & R \cat S
      \ar[d]^{RG} \ar[rr]^{X} & & R\C\rlap{ .}\\
      \Sigma \mathsf T \ar[r]_-{\eta} & R \cat T
     }
  \end{equation*}
  Here, the maps labelled $\eta$ exhibit $R\cat S$ and $R\cat T$ as
  $L\Sigma \mathsf S$ and $L \Sigma \mathsf T$.  Let $\C$ be a
  category with finite powers, and let $X$ as displayed above be an
  $\cat S$-model in $\C$, with $X\eta \colon \Sigma \mathsf S \to R
  \C$ the corresponding $\mathsf S$-algebra. Then the
  following are equivalent:
\begin{enumerate}[(i)]
\item The Kan extension  $\Lan_{RG}(X) \colon R\cat T \to R\C$ exists;
\item The Kan extension $\Lan_{\Sigma F}(X\eta) \colon R\mathsf T \to
  R\C$ exists;
\item The ordinary functor
  \begin{equation*}
    D \colon G \downarrow A \xrightarrow{\pi_1} \cat S \xrightarrow{VX} \C
\end{equation*}
(where
  $A$ is the distinguished object of $\cat T$) admits a colimit which
  distributes over finite powers;
\item The unenriched Kan extension $\Lan_{G}(VX) \colon \cat T
  \to \C$ exists and is a finite-power-preserving functor.
\end{enumerate}
\end{Prop}
Once again, the hypotheses of this proposition will always be
satisfied when $\C$ is cartesian closed and cocomplete, so in
particular when $\C = \cat{Set}$.
\begin{proof}
  Assume (i). By~\cite[Theorem~5.35]{Kelly1982Basic}, we have $X \cong
  \Lan_\eta(X\eta)$, and so $\Lan_{RG}(X) \cong \Lan_{RG.\eta}(X\eta)
  \cong \Lan_{\eta.\Sigma F}(X\eta)$. As $\eta \colon \Sigma \mathsf T
  \to R \cat T$ is fully faithful, it follows that $(\Lan_{RG}X).\eta
  \cong \Lan_{\Sigma F}(X\eta)$, as required for (ii). Conversely,
  given (ii), the left Kan extension of $\Lan_{\Sigma F}(X\eta)$ along
  $\eta \colon \Sigma \mathsf T \to R \cat T$ exists, again
  by~\cite[Theorem~5.35]{Kelly1982Basic}, and is isomorphic to
  $\Lan_{RG}(X)$ as above, giving (i).

  We next show that (ii) $\Leftrightarrow$ (iii). By definition,
  $\Lan_{\Sigma F}(X\eta)$ has its value at the unique object of
  $\Sigma \mathsf T$ given by the relative tensor $T \otimes_{\Sigma
    \mathsf S} (X\eta)$; here $T$ is regarded as a right $\Sigma
  \mathsf S$-module via the action
  \[
  T \otimes S \xrightarrow{1 \otimes F} T \otimes T \xrightarrow{\mu}
  T\rlap{ .}
  \]
  Under the isomorphism of Proposition~\ref{prop:nerve}, this right
  module is easily seen to correspond to the presheaf $\cat T(G, A)
  \in [\cat S^\op, \cat{Set}]$ (with $A$ the distinguished object of
  $\cat T$); whence, by Proposition~\ref{prop:fsetreltensors},
  the relative tensor $T \otimes_{\Sigma \mathsf S} (X\eta)$ exists if
  and only if the conditions in (iii) hold.

  Finally, we show (iii) $\Leftrightarrow$ (iv). The value at $A^n$ of
  $\Lan_G(VX)$ is given by the (unenriched) weighted colimit $\cat
  T(G,A^n) \star VX$; but $\cat T(G,A^n) \cong (\cat T(G,A))^n$ and
  so, repeating the calculation in~\eqref{eq:calculation}, this
  weighted colimit can be computed as the conical colimit of $D^n
  \colon (G \downarrow A)^n \to \C$. Thus the existence and
  finite-power-preservation of $\Lan_G(VX)$ is equivalent to the
  existence and distributivity of $\mathrm{colim}\ D$, as required.
\end{proof}

 \bibliographystyle{acm}

\bibliography{bibdata}
 
\end{document}